\documentclass[11pt, reqno, a4paper]{amsart}
\usepackage{amsthm,amsfonts,amsmath,amssymb}
\usepackage[utf8]{inputenc}
\usepackage[usenames]{color}
\usepackage{mathrsfs}

\usepackage[pdftex,bookmarks=true]{hyperref}

\usepackage{enumerate}
\usepackage{tikz-cd}
\usepackage{graphicx}

\numberwithin{equation}{section}

\theoremstyle{definition}
\newtheorem{definition}{Definition}[section]

\theoremstyle{remark}
\newtheorem{remark}[definition]{Remark}

\theoremstyle{plain}
\newtheorem{theorem}[definition]{Theorem}
\newtheorem{result}[definition]{Result}
\newtheorem{lemma}[definition]{Lemma}
\newtheorem{proposition}[definition]{Proposition}
\newtheorem{corollary}[definition]{Corollary}
\newtheorem{fact}[definition]{Fact}

\newcommand{\eps}{\varepsilon}

\newcommand{\zt}{\zeta}

\newcommand\partl[2]{\frac{\partial{#1}}{\partial{#2}}}
\newcommand{\dbar}{\overline{\partial}}

\newcommand{\bdy}{\partial}
\newcommand{\OM}{\Omega}
\newcommand{\disk}{\mathbb{D}}

\newcommand{\hpln}{\boldsymbol{{\sf H}}}
\newcommand{\Nb}{\mathcal{W}}
\newcommand{\patch}{\mathcal{V}}
\newcommand{\dee}{\mathcal{D}}

\newcommand{\smoo}{\mathcal{C}}
\newcommand{\hol}{\mathscr{O}}

\newcommand{\re}{{\sf Re}}
\newcommand{\im}{{\sf Im}}
\newcommand{\symq}{\pi_{\!{\raisebox{-1.5pt}{$\scriptstyle{Sym}$}}}}
\newcommand\elsym[1]{\sigma_{{#1}}}

\newcommand{\bcdot}{\boldsymbol{\cdot}}

\newcommand{\excep}{\mathscr{E}}
\newcommand{\smu}{u^{\raisebox{-1pt}{$\scriptstyle\infty$}}}

\newcommand{\Cn}{\mathbb{C}^n}

\newcommand{\C}{\mathbb{C}}

\newcommand{\cplx}{\mathbb{C}}

\newcommand{\atlas}{\mathfrak{A}}
\newcommand{\Sym}{{\rm Sym}}

\newcommand{\N}{\mathbb{N}}
\newcommand{\Z}{\mathbb{Z}}

\begin{document}

\title[Proper holomorphic maps onto symmetric products]{Proper holomorphic
mappings onto symmetric \\ products of a Riemann surface}

\author[G.~Bharali]{Gautam Bharali}
\address{Department of Mathematics, Indian Institute of Science, Bangalore 560012, India}
\email{bharali@iisc.ac.in}

\author[I.~Biswas]{Indranil Biswas}
\address{School of Mathematics, Tata Institute of Fundamental Research, 1 Homi Bhabha
  Road, Mumbai 400005, India}
\email{indranil@math.tifr.res.in}

\author[D.~Divakaran]{Divakaran Divakaran}
\address{Department of Mathematics, Indian Institute of Science Education and Research Bhopal,
  Bhopal 462066, India}
\email{divakaran@iiserb.ac.in}

\author[J.~Janardhanan]{Jaikrishnan Janardhanan}
\address{Department of Mathematics, Indian Institute of Technology Madras, Chennai 600036, India}
\email{jaikrishnan@iitm.ac.in}

\keywords{Proper holomorphic maps, Riemann surfaces, symmetric product, taut manifolds}
\subjclass[2010]{Primary 14J50, 32H35; Secondary 32Q45}

\begin{abstract}
We show that the structure of proper holomorphic maps between the $n$-fold symmetric
products, $n\geq 2$, of a pair of non-compact Riemann surfaces $X$ and $Y$, provided
these are reasonably nice, is very rigid. Specifically, any such map is determined by a proper
holomorphic map of $X$ onto $Y$. This extends existing results concerning bounded planar
domains, and is a non-compact analogue of a phenomenon observed in symmetric
products of compact Riemann surfaces. Along the way, we also provide a condition for the 
complete hyperbolicity of all $n$-fold symmetric products of a non-compact Riemann surface.
\end{abstract}

\maketitle

\vspace{-0.5cm}
\section{Introduction}\label{s:intro}

It is well known that the $n$-fold symmetric product of a Riemann surface, $n\geq 2$,
is an $n$-dimensional complex manifold. One has a precise
description of all proper holomorphic maps between $n$-fold products of bounded
planar domains\,---\,provided by the Remmert--Stein Theorem
\cite{remmertStein1960propMap}  (also see
\cite[pp.\!~71--78]{narasimhan1971SCV})\,---\,and, more recently, of finite proper
holomorphic maps between products of Riemann surfaces; see
\cite{janardhanan2014properMap}, for instance. It is therefore natural to investigate
the structure of such maps between symmetric products of Riemann surfaces. To this
end, we are motivated by the following result of Edigarian and Zwonek
\cite{edigarian2005geometry} (the notation used will be explained below):

\begin{result}[paraphrasing {\cite[Theorem~1]{edigarian2005geometry}}]\label{r:e-z}
  Let $\mathbb{G}^n$ denote the $n$-dimensional symmetrized polydisk and
  let $f : \mathbb{G}^n\to \mathbb{G}^n$ be a proper holomorphic map. Then,
  there exists a finite Blaschke product $B$ such that
  \[
    f\big(\pi^{(n)}(z_1,\dots, z_n)\big)\,=\,\pi^{(n)}\big(B(z_1),\dots, B(z_n)\big) \; \;
    \forall (z_1,\dots, z_n)\in \disk^n.
  \]
\end{result}

In this result, and in what follows, we denote the open unit disk with centre
$0\in \C$ by $\disk$. Throughout this paper $\sigma_j$, $j = 1,\dots, n$, will denote the
elementary symmetric polynomial of degree $j$ in $n$ indeterminates (when there is
no ambiguity, we shall\,---\,for simplicity of notation\,---\,suppress the parameter
$n$). The map $\pi^{(n)} : \C^n\to \C^n$ is defined as:
\[
  \pi^{(n)}(z_1,\dots, z_n)\,:=\,\big(\sigma_1(z_1,\dots, z_n), \sigma_2(z_1,\dots, z_n),
  \dots, \sigma_n(z_1,\dots, z_n)\big), \; \; (z_1,\dots, z_n)\in \C^n.
\]
The {\em symmetrized polydisk}, $\mathbb{G}^n$, is defined as
$\mathbb{G}^n := \pi^{(n)}(\disk^n)$. It is easy to see that $\mathbb{G}^n$ is a domain
in $\C^n$, whence $\mathbb{G}^n$ is a holomorphic embedding of the $n$-fold symmetric
product of $\disk$ into $\C^n$.

Given a Riemann surface $X$, we shall denote its $n$-fold symmetric product by
$\Sym^n(X)$.  The complex structure on $X$ induces a complex structure on
$\Sym^n(X)$, which is described in brief in Section~\ref{s:symm_prelim} below.

In this paper, we shall extend Result~\ref{r:e-z}\,---\,see
Corollary~\ref{c:two_symmprods} below\,---\,to proper holomorphic maps between
the $n$-fold symmetric products of certain non-compact Riemann surfaces. At this
juncture, the reader might ask whether there is an analogous generalization of
Result~\ref{r:e-z} to $n$-fold symmetric products of {\em compact} Riemann
surfaces. Before we answer this question, we state the following result and 
note that Corollary~\ref{c:two_symmprods} is its non-compact analogue. Indeed,
the following result (the notation therein is explained below) was among our
motivations for the investigation in this paper.

\begin{fact}[an {\bf adaptation} of the results in \cite{cilibertoSernesi:symmetric93}
by Ciliberto--Sernesi]\label{f:compact_isom}
Let $X$ and $Y$ be compact Riemann surfaces with
${\rm genus}(X) = {\rm genus}(Y) = g$, where $g > 2$. Let
$F : \Sym^n(X)\to \Sym^n(Y)$ be a surjective holomorphic map, where
$n = 1, 2, 3,\dots, 2g-3$, $n\neq g-1$. Then:
\begin{itemize}
  \item[(1)] $X$ is biholomorphic to $Y$;
  
  \item[(2)] The map $F$ is a biholomorphism; and
  
  \item[(3)] There exists a biholomorphic map $\phi : X\to Y$ such that
  \[
    F(\langle x_1,\dots, x_n\rangle)\,=\,\langle\phi(x_1),\dots \phi_n(x_n)\rangle
    \; \; \forall \langle x_1,\dots, x_n\rangle\in \Sym^n(X).
  \]
\end{itemize}
\end{fact}
\noindent{In the above, and in what follows,
we denote by $\langle x_1,\dots, x_n\rangle$ the orbit of 
$(x_1,\dots, x_n)\in X^n$ under the $S_n-$action on $X^n$ that permutes
the entries of $(x_1,\dots, x_n)$. The map
\[
  X^n\ni (x_1,\dots, x_n)\,\longmapsto\,\langle x_1,\dots, x_n\rangle
  \; \;\forall (x_1,\dots, x_n)\in X^n
\]
will be denoted by $\symq^X$. When there is no ambiguity, we shall drop
the superscript.

\begin{remark}
The paper \cite{cilibertoSernesi:symmetric93} does not contain a 
statement of Fact~\ref{f:compact_isom} in the specific form given above.
Therefore, we provide a justification. We first consider the
case $1\leq n\leq g-2$. A special case of a theorem by Martens
\cite{martens:torelli63} gives us (1) and (3)
of Fact~\ref{f:compact_isom}, assuming that $F$ is a biholomorphism.
In \cite[Section~2]{cilibertoSernesi:symmetric93}, Ciliberto and Sernesi give a
different proof of Martens's theorem. It is straightforward to check that the
proof in \cite{cilibertoSernesi:symmetric93} yields (1)--(3) above\,---\,taking
$1\leq n\leq g-2$\,---\,even when $F$ is just a surjective holomorphic map.

Now consider the case $g\leq n\leq 2g-3$. This time, the first two
paragraphs following the heading ``Proof of Theorem~(1.3)''
in \cite{cilibertoSernesi:symmetric93} give us
(1) and (3)
of Fact~\ref{f:compact_isom}, assuming again that $F$ is a biholomorphism.
Again, it is straightforward to check that the requirement that $F$ be a
biholomorphism is not essential. The argument in those paragraphs gives
us (1)--(3) above\,---\,taking
$g\leq n\leq 2g-3$\,---\,even when $F$ is just a surjective
holomorphic map.  \hfill $\blacktriangleleft$ 
\end{remark}

The restrictions on the pair $(g, n)$ in Fact~\ref{f:compact_isom} are
essential. It is classically known that there exist nonisomorphic compact
Riemann surfaces of genus 2 having isomorphic Jacobians, hence isomorphic
$2$-fold symmetric products. 
Next, consider a non-hyperelliptic compact
Riemann surface $X$ of  genus 3. Given any $\langle x_1, x_2\rangle \in \Sym^2(X)$,
there is a unique
point $\langle y_1, y_2\rangle \in \Sym^2(X)$ such that the divisor
$(x_1 + x_2 + y_1 + y_2)$ represents the
holomorphic cotangent bundle. The automorphism of $\Sym^2(X)$ given by
$\langle x_1,x_2\rangle\mapsto \langle y_1, y_2\rangle$ is {\em not} given
by any automorphism of $X$ (here, $g = 3$, $n = 2$, whence $n = g-1$).
Furthermore, we expect
any generalization of Fact~\ref{f:compact_isom} to be somewhat intricate 
because, among other things:
\begin{itemize}
  \item Any generalization wherein ${\rm genus}(X) \neq {\rm genus}(Y)$
  will place restrictions on the pair $({\rm genus}(X), {\rm genus}(Y))$ owing to
  the Riemann--Hurwitz formula.
  
  \item The geometry of $\Sym^n(X)$ varies considerably depending on
  whether $1\leq n\leq  {\rm genus}(X)-1$ or $n\geq {\rm genus}(X)$.
\end{itemize}
In short, any generalization of Fact~\ref{f:compact_isom} would
rely on techniques very different from those involved in
proving Corollary~\ref{c:two_symmprods}. Thus, we shall address
the problem of the structure of surjective holomorphic maps in the
compact case in forthcoming work.
   
We now focus on $n$-fold symmetric products of {\em non-compact}
Riemann surfaces. We should
mention here that Chakrabarti and Gorai have extended Result~\ref{r:e-z}
to $n$-fold symmetric products of bounded planar domains in 
\cite{debraj2015function} 
Their result as well as Result~\ref{r:e-z} rely on an interesting
adaptation\,---\,introduced in \cite{edigarian2005geometry}\,---\,of an
argument by Remmert--Stein. The latter argument relies on two
essential analytical ingredients:
\begin{itemize}
  \item[$(i)$] The ability to extract subsequences\,---\,given an auxiliary
  sequence constructed from the given proper map\,---\,that converge
  locally uniformly; and
  
  \item[$(ii)$] A vanishing-of-derivatives argument that stems from the
  mean-value inequality.
\end{itemize}
These ingredients continue to be relevant when planar domains are replaced
by Riemann domains and, indeed, parts of our proofs emulate the argument
in \cite{edigarian2005geometry}.

However, our proofs of the theorems below
do {\bf not} reduce to a mere application of Result~\ref{r:e-z} to appropriate
coordinate patches. An explanation of this is presented in the paragraph that
follows \eqref{e:const} below. Equally significantly, we need to identify a class
of Riemann surfaces $X$ for which some form of the ingredient $(i)$ above is
available for $\Sym^n(X)$, $n\geq 2$. This is the objective of our first
theorem\,---\,which might also be of independent interest.      
   
\begin{theorem}\label{t:symm_taut}
Let $X$ be a connected bordered Riemann surface with
$\smoo^2$-smooth boundary. Then $\Sym^n(X)$ is Kobayashi
complete, and hence taut, for each $n\in \Z_+$.
\end{theorem}

We must clarify that in this paper the term {\em connected bordered
Riemann surface with $\smoo^2$-smooth boundary} refers to a non-compact
Riemann surface $X$ obtained by excising from a compact Riemann $S$
a finite number of closed, pairwise disjoint disks $D_1,\dots, D_m$ such that
$\bdy{D_j}$, $j = 1,\dots, m$, are $\smoo^2$-smooth. The complex
structure on $X$ is the one it inherits from $S$: i.e., a holomorphic chart
of $X$ is of the form
$(\varphi, U\!\setminus\!(D_1\sqcup\dots \sqcup D_m))$,
where $(\psi, U)$ is a holomorphic chart of $S$ and $\varphi$
is the restriction of $\psi$ to $U\!\setminus\!(D_1\sqcup\dots \sqcup D_m)$.

The ingredients $(i)$ and $(ii)$ above allow us to analyse
proper holomorphic maps between a product manifold of dimension
$n$ and an $n$-fold symmetric product, where the factors of the
product manifold need not necessarily be the same. This is formalised
by our next theorem. A similar result
is proved in \cite{debraj2015function} where the factors of the
products involved are bounded planar domains.
Corollary~\ref{c:two_symmprods} is obtained as an easy
consequence of the following:  

\begin{theorem}\label{t:bord}
  Let $X=X_1\times \cdots \times X_n$ be
  a complex manifold where each $X_j$ is a connected non-compact Riemann surface
  obtained by excising a non-empty indiscrete set from a compact Riemann surface $R_j$.
  Let $Y$ be a connected
  bordered Riemann surface with $\smoo^2$-smooth boundary. Let
  $F: X \to \Sym^n(Y)$ be a 
  proper holomorphic map. Then, there exist proper holomorphic maps
  $F_j : X_j \to Y$, $j = 1,\dots, n$, such that
  \[
    F(x_1,\dots, x_n)\,=\,\symq\circ \big(F_1(x_1),\dots, F_n(x_n)\big)
    \; \; \; \forall (x_1,\dots, x_n)\in X.
  \]
\end{theorem}

The complex structure on each of the factors $X_1,\dots, X_n$ has a
description analogous to the one given above for bordered Riemann surfaces.

Finally, we can state the corollary alluded to above. Observe that it is the
analogue, in a non-compact setting, of Fact~\ref{f:compact_isom}. It also
subsumes Result~\ref{r:e-z}: recall that the proper holomorphic self-maps of
$\disk$ are precisely the finite Blaschke products.   

\begin{corollary}\label{c:two_symmprods}
  Let $X$ be a connected non-compact Riemann surface obtained by
  excising a non-empty indiscrete set
  from a compact Riemann surface $R$, and let $Y$ be a connected
  bordered Riemann surface with $\smoo^2$-smooth boundary. Let
  $F: \Sym^n(X) \to \Sym^n(Y)$ be a 
  proper holomorphic map. Then, there exists a proper holomorphic map
  $\phi : X \to Y$ such that
  \[
    F\circ \symq^X(x_1,\dots, x_n)\,=\,\symq^Y\big(\phi(x_1),\dots, \phi(x_n)\big) \; \; \;
   \forall (x_1,\dots, x_n)\in X^n.
  \]
\end{corollary}

This corollary follows immediately from Theorem~\ref{t:bord} since 
$F\circ \symq^X : X^n\to \Sym^n(Y)$ is proper.

We conclude this section with an amusing observation that follows from
Corollary~\ref{c:two_symmprods}. We first make an explanatory
remark. It is well known that if $M_1$ and $M_2$ are two
non-compact complex manifolds of the same dimension and
$F: M_1\to M_2$ is a proper holomorphic map, then there exists a
positive integer $\mu$ such that, for any generic point $p\in M_2$,
$F^{-1}\{p\}$ has cardinality $\mu$. We call this number the {\em multiplicity}
of $F$, which we denote by ${\rm mult}(F)$.

\begin{corollary}
  Let $X$ and $Y$\,---\,a pair of connected non-compact Riemann
  surfaces\,---\,be exactly as in Corollary~\ref{c:two_symmprods}.
  If $F: \Sym^n(X) \to \Sym^n(Y)$ is a proper holomorphic map,
  then ${\rm mult}(F)$ is of the form $d^n$, where $d$ is some
  positive integer.
\end{corollary}

\section{Preliminaries about the symmetric products}\label{s:symm_prelim}

In this section we shall give a brief description, given a 
Riemann surface $X$, of the complex structure on $\Sym^n(X)$, $n\geq 2$,
that makes it a complex manifold. We shall use the notation introduced in
Section~\ref{s:intro}. Given this notation:
\begin{itemize}
  \item Recall that $\langle x_1,\dots, x_n \rangle := \symq(x_1,\dots,x_n)$,
  \item  Given subsets $V_j\subseteq X$ that are open, let us write:
  \[
    \langle V_1,\dots, V_n \rangle := \left\{\langle x_1,\dots x_n\rangle : x_j \in V_j,
    \; j = 1, \dots, n\right\}.
  \]
\end{itemize}
Since $\Sym^n(X)$ is endowed with the quotient topology relative to $\symq$,
$\langle V_1,\dots, V_n\rangle$ is, by definition, open in $\Sym^n(X)$.
\smallskip

$\Sym^n(X)$ is endowed with a complex structure as follows.
Given a point $p\in \Sym^n(X)$, $p = \langle p_1,\dots p_n \rangle$,
choose a holomorphic chart $(U_j, \varphi_j)$ of $X$ at $p_j$,
$j = 1,\dots, n$, such that
\[
  U_j\cap U_k = \emptyset \ \ \text{if $p_j\neq p_k$} \qquad \text{and}
  \qquad U_j = U_k \ \ \text{if $p_j = p_k$}.
\]
The above choice of local charts ensures that the map 
$\varPsi_p: \langle U_1,\dots, U_n\rangle\to \C^n$ given by
\[
 \varPsi_p: \langle x_1,\dots, x_n\rangle\,\longmapsto\,\big(\elsym{1}(\varphi_1(x_1),\dots,
 		\varphi_n(x_n)),\dots, \elsym{n}(\varphi_1(x_1),\dots,\varphi_n(x_n))\big),
\]
(where $\sigma_1,\dots, \sigma_n$ are the elementary symmetric polynomials
that were introduced in Section~\ref{s:intro}) is a homeomorphism. This
follows from the Fundamental Theorem of Algebra. 
The collection of such charts
$(\langle U_1,\dots, U_n\rangle, \varPsi_p)$ produces a holomorphic
atlas on $\Sym^n(X)$. We shall call such a chart a {\em model coordinate
chart at $p\in \Sym^n(X)$}.

Finally, let $Z$ be a compact Riemann surface, $X\varsubsetneq Z$ be an embedded
open complex submanifold of $Z$, and let $\atlas(Z)$ denote the complex structure
on $Z$. Then, since\,---\,for any point
$p\in X$\,---\,there is a chart $(U, \varphi)\in \atlas(Z)$ such that $U\subset X$,
the above discussion shows that
$\Sym^n(X)$ is an embedded complex submanifold of $\Sym^n(Z)$. We refer the
reader to \cite{whitney1972varieties} for details.
\smallskip

\section{Hyperbolicity and its consequences}

The proof of Theorem~\ref{t:bord} will require several results about
holomorphic mappings into Kobayashi hyperbolic spaces. We summarize the
relevant results
in this section. An encyclopedic reference for the results in this is section
is \cite{kobayashi98hyp}.

In the theory of holomorphic functions of one variable, the behaviour of 
holomorphic functions near an isolated singularity is well-studied.
Among the important results in this area are the famous theorems of Picard.
A consequence of Picard's big theorem is that any
meromorphic mapping on $\disk \setminus \{0\}$ that misses three
points automatically extends to a meromorphic function defined on
the whole of $\disk$. One of the proofs of Picard's theorem
relies on the fact that the sphere with three points
removed is a hyperbolic Riemann surface. This perspective allows one
to generalize the aforementioned extension theorem to holomorphic
mappings into Kobayashi hyperbolic spaces. To this end, we need a definition.

\begin{definition}
  Let $Z$ be a complex manifold and let $Y$ be a relatively compact complex 
  submanifold of $Z$. We call a point $p \in \overline{Y}$ a 
  \textit{hyperbolic point} if every $Z$-open neighbourhood $U$ of $p$ contains a
  smaller neighbourhood $V$ of $p,\ \overline{V} \subset U$, such that
  \begin{equation}\label{e:hyp_imb}
    K_{Y}(\overline{V} \cap Y, Y\setminus U) := \inf\{K_Y(x,y) : x \in
    \overline{V} \cap Y, y \in Y \setminus U\} > 0,
  \end{equation}
  where $K_{Y}$ denotes the Kobayashi pseudo-distance on $Y$. We say that $Y$ is
  \emph{hyperbolically embedded} in $Z$ if every point of $\overline{Y}$ is a
  hyperbolic point.  
\end{definition}

The following result is an example of an extension result in higher dimensions
of the type alluded to above.


\begin{result}[Kiernan {\cite[Theorem~3]{kiernan72extension}}]
\label{r:kwack}
  Let $X$ be a complex space and let $\excep \subset X$ be a closed complex subspace.
  Let $Y$ be a complex manifold that is hyperbolically embedded in a complex
  manifold $Z$. Then every holomorphic map $f:X\setminus\excep \to Y$ extends
  to a meromorphic map $\widetilde{f}: X\to Z$.
\end{result}

This result will play a role in the final stages of proving
Theorem~\ref{t:bord}. To this end, we would also need\,---\,naturally, given the
statement of Theorem~\ref{t:bord}\,---\,conditions under which  
a meromorphic map between complex spaces is actually
holomorphic. One situation where this happens is when the complex spaces
are manifolds and the target space is Kobayashi hyperbolic. 

\begin{result}[Kodama \cite{kodama79bimero}]\label{r:merhyp}
  Let $f : X\to Y$ be a meromorphic map, where $X$ is a complex manifold and
  $Y$ is a Kobayashi hyperbolic manifold. Then $f$ is holomorphic.
\end{result}

The following lemma enables us\,---\,as we shall see in Section~\ref{s:bord}\,---\,to
use the preceding results in our specific set-up.

\begin{lemma}\label{l:embed}
  Let $Y$, a non-compact Riemann surface, be as in
  Theorem~\ref{t:bord} and let $S$ be the compact connected Riemann
  surface from which $Y$ is obtained by excising a finite number
  of closed disks. Then
  $\Sym^n(Y)$ is hyperbolically embedded in $\Sym^n(S)$.
\end{lemma}
\begin{proof}
  Let $\mathcal{W} \subset S$ be another
  connected bordered Riemann surface with $\smoo^2$-smooth boundary
  such that
  $\overline{Y}\subset \mathcal{W}$. By Theorem~\ref{t:symm_taut}, $\Sym^n(Y)$ 
  and $\Sym^n(\mathcal{W})$ are both Kobayashi complete. In particular,
  $K_Y$ and $K_{\mathcal{W}}$ are {\bf distances}. It follows from the
  discussion at the end of Section~\ref{s:symm_prelim} that
  $\Sym^n(Y)$ and $\Sym^n(\mathcal{W})$ are embedded submanifolds
  of $\Sym^n(S)$. Observe that it suffices to show
  that each $p\in \bdy{\Sym^n(Y)}$ is holomorphically embedded in $S$.
  Fix a point $p\in \bdy{\Sym^n(Y)}$. Given any $S$-open neighbourhood $U$ of
  $p$, we choose a neighbourhood $V$ of $p$ such that $\overline{V}\subset U$ and
  $\overline{V}\subset \Sym^n(\mathcal{W})$. For any $x\in \overline{V}\cap Y$
  and $y\in Y\setminus U$, we have
  $K_Y(x, y)\,\geq\,K_{\mathcal{W}}(x,y) > 0$.
  We know that $\overline{V}\cap \overline{Y}$ and $\overline{Y\setminus U}$ are
  compact in $\mathcal{W}$. Thus, the inequality in \eqref{e:hyp_imb} follows from the
  last inequality. 
\end{proof}

As the proof of the above lemma shows, Theorem~\ref{t:symm_taut} is
an essential ingredient in the proof of Theorem~\ref{t:bord}. In the remainder
of this section, we shall present some prerequisites for proving
Theorem~\ref{t:symm_taut}. We begin with a couple of definitions.

\begin{definition}
Let $Z$ be a complex manifold and $Y\subset Z$ be a connected open subset of $Z$.
Let $p\in \bdy{Y}$. We say that {\em $p$ admits a weak peak function for $Y$}
if there exists a continuous function $f_p: \overline{Y}\to \cplx$ such that
$f|_Y$ is holomorphic,
\[
  f_p(p) = 1 \qquad \text{and} \qquad |f_p(y)| < 1 \; \; \forall y\in Y.
\]
We say that {\em $p$ admits a local weak peak function for $Y$} if $p$
admits a weak peak function for $Y\cap U_p$, where $U_p$ is some open neighbourhood
(in $Z$) of $p$.
\end{definition}

In what follows, given a complex manifold $X$, $C_X$ will denote the Carath{\'e}odory
pseudo-distance on $X$. The term {\em Carath{\'e}odory hyperbolic} has a
meaning analogous to that of the term Kobayashi hyperbolic. Furthermore, we say
that $X$ is {\em strongly $C_X$-complete} if $X$ is Carath{\'e}odory hyperbolic
and each closed ball in $X$, with respect to the distance $C_X$, is compact.

We shall also need the following:

\begin{result}\label{R:s-Car_complt}
Let $Z$ be a Stein manifold and $Y\subset Z$ a relatively compact connected open
subset of $Z$. If each point of $\bdy{Y}$ admits a weak peak function for $Y$,
then $Y$ is strongly $C_Y$-complete.
\end{result}

The above result has been established with $Z = \cplx^n$ and $Y$ a bounded domain
in $\cplx^n$ in \cite[Theorem~4.1.7]{kobayashi98hyp}. Its proof applies
{\em mutatis mutandis} for $Y$ and $Z$ as in Result~\ref{R:s-Car_complt} (that the class of
bounded holomorphic functions on $Y$ separates points is routine to show
with our assumptions on the pair $(Y, Z)$).
\smallskip

\section{The proof of Theorem~\ref{t:symm_taut}}

Before we provide a proof, some remarks on notation are in order. For simplicity of notation,
in this section ({\bf unlike} in subsequent sections), the symbol $D_j$, $j\in \N$, will denote
a closed topological disk. Given non-empty open subsets $A$ and $B$ of the Riemann
surface $S$ (explained below), we shall denote
the relation $\overline{A}\subset B$ (especially when there is a sequence of such relations)
as $A\Subset B$, where the closure is taken in $S$.
 
\begin{proof}[The proof of Theorem~\ref{t:symm_taut}]
  We begin with the following
  \smallskip
 
  \noindent{{\bf Claim.} {\em Each $y\in \bdy{X}$ admits a weak peak function for $X$.}}
  \vspace{0.5mm}
  \noindent{Each of the individual ingredients in this construction is classical, so we shall be brief.
  Fix $y\in \bdy{X}$. Let $S$ be the compact Riemann surface such that
  \[
    X = S\setminus \big(D_1\sqcup\dots \sqcup D_m\big),
  \]
  where each $D_j$ is a {\bf closed} topological disk with $\smoo^2$-smooth boundary. We may
  assume without loss of generality that $y\in \bdy{D_1}$. Let us write
  $X^* = S\setminus (\Delta_1\sqcup\dots \sqcup \Delta_m)$, where each $\Delta_j$,
  $j = 1,\dots, m$, is a closed topological disk such that
  \[
    \Delta_j  \varsubsetneq (D_j)^\circ, \; \; j = 2,\dots, m.
  \]
  Before we describe $\Delta_1$, let us choose a holomorphic chart $(U, \psi)$
  centered at $y$ such that:
  \begin{itemize}
    \item $\psi : (U, y) \longrightarrow (\disk, 0)$,
    \item $\psi^{-1}((0, 1]\,)\subset S\setminus \overline{X}$, and
    \item $U$ is {\bf so small} that
    $\psi^{-1}(\disk\cap \overline{D(1; 1)}\,)\cap \overline{X} = \{y\}$.
  \end{itemize}
  The last requirement is possible because $\bdy{X}$ is of class $\smoo^2$.
  It is easy to construct a {\em local} peak function $\phi$ at $y$ for $X$ that is, in
  fact, holomorphic on $U$ and such that
  \begin{align}
    |\phi(x)| &< 1 \; \; \forall x\in U\!\setminus\!\psi^{-1}(\disk\cap \overline{D(1; 1)}\,),
    \label{E:small} \\
    |\phi(x)| &> 1 \; \; \forall x\in \psi^{-1}(\disk\cap D(1; 1)), \notag \\
    |\phi(x)| &= 1 \; \; \forall x\in \psi^{-1}(\disk\cap \bdy{D(1; 1)}) \ \;\text{with}
    \; \ \phi^{-1}\{1\} = \{y\} \notag
  \end{align}
  (the interested reader is referred to the proof of Proposition~\ref{p:INDEP}
  for further details). Let $\Delta_1\varsubsetneq (D_1)^\circ$ and be such that
  $\Delta_1\cap U\neq \emptyset$ and
  $\bdy{\Delta_1}$ intersects $\bdy{U}$ at exactly two points. In fact, we can
  choose $\Delta_1$ such that, in addition to these properties, 
  $\bdy{\Delta_1}$ also intersects $\psi^{-1}(\{\zt\in \C: |\zt| = 1-\eps\})$ in exactly
  two points for some positive $\eps\ll 1$.
  Pick two $S$-open neighbourhoods, $V_1$ and $V_2$, of $y$ such that
  \[
    V_1\Subset V_2\Subset U \quad\text{and} 
    \quad \psi^{-1}((1-\eps)\disk\cap \overline{D(1; 1)}\,)\cap X^*\subset V_1.
  \]
  Let $\chi_1, \chi_2 \longrightarrow [0,1]$ be two functions in $\smoo^\infty(X^*)$ with
  \begin{align*}
    \left.\chi_1\right|_{V_1\cap X^*}\equiv 1 \quad\text{and} \quad
    \left.\chi_1\right|_{X^*\setminus V_2}\equiv 0, \\
    \left.\chi_2\right|_{V_2\cap X^*}\equiv 1 \quad\text{and} \quad
    \left.\chi_2\right|_{X^*\setminus U}\equiv 0.
  \end{align*}
  Finally, consider the function:
  \[
    G(x):=\begin{cases}
		   (1-\phi(x))\chi_2(x), &\text{if $x\in (X^*\cap U)$}, \\
		   0, &\text{if $x\in (X^*\setminus U)$}.
		\end{cases}
  \]}
  
  In what follows, it will be understood that any expression of the form $\Psi/G$
  is {\bf $\boldsymbol{0}$ by definition} outside ${\sf supp}(\Psi)$. Define the $(0,1)$-form
  $\omega\in \Gamma\big(T^{\boldsymbol{*}\,(0,1)}S|_{X^*}\big)$ as
  \[
    \omega = \frac{\dbar\chi_1}{G}
  \]
  By construction, $\omega$ is of class $\smoo^\infty$, vanishes on
  $(X^*\cap \overline{V}_1)\cup (X^*\!\setminus V_2)$, and
  \begin{equation}\label{E:closed}
    x\in X\cap (V_2\setminus \overline{V}_1)\,\Longrightarrow\,\omega(x)
    = \left.\dbar\left(\frac{\chi_1}{1-\phi}\right)\right|_x.
  \end{equation}
  By the Behnke--Stein theorem \cite{behnkeStein:eaFRF48}, $X^*$ is Stein. Thus,
  it admits a solution to the $\dbar$-problem
  \begin{equation}\label{E:d-bar}
    \dbar u\,=\,\omega \; \ \text{on $X^*$.}
  \end{equation}
  Furthermore, it is a classical fact that there exists a solution, say $\smu$, to
  \eqref{E:d-bar} of class $\smoo^\infty(X^*)$.
  Write $u_y := \left.\smu\right|_{\overline{X}}$.
  As $X\Subset X^*$, $u_y$ is bounded. Thus\,---\,by subtracting a large
  positive constant if necessary\,---\,we may assume that $\re(u_y) < 0$ on 
  $\overline{X}$. Observe that, by \eqref{E:closed}, \eqref{E:d-bar} and the
  construction of $G$,
  \[
    \big(-\!(\chi_1/G) + u_y\big)^{-1}\in \hol(X).
  \]
  By \eqref{E:small} and by our adjustment of $\re(u_y)$, we have
  \[
    \re\big( (-(\chi_1/G) + u_y )^{-1} \big)(x) < 0 \; \; \forall x\in X.
  \]
  From this, it is easy to check that $f_y(x) := e^{(\,1/(-(\chi_1/G) + u_y))(x)}$,
  $x\in \overline{X}$, is a weak peak function at $y$ for $X$. Hence our claim.
  \smallskip

  In this paragraph, we assume that $n\geq 2$.
  Let us pick a point $\langle y_1,\dots, y_n\rangle\in \bdy\Sym^n(X)$.
  It is routine to see that $\symq$ is a proper map. Thus, we may assume
  without loss of generality that $y_1\in \bdy{X}$. Our Claim above gives
  us a weak peak function for $X$ at $y_1$: call it $f$. Set
  \[
    {\sf h}(z)\,:=\,\frac{1+z}{1-z},
  \]
  which maps $\disk$ biholomorphically to the open right half-plane $\hpln_+$ and
  maps $(0, 1)\longmapsto (1, +\infty)$. Let $(\bcdot)^{1/n}$ denote
  the holomorphic branch on $\hpln_+$ of the $n$-th root such that
  \begin{equation}\label{E:branch}
    z^{1/n} \in \big\{w\in \cplx : \re(w) > 0, \ |\im(w)| < \arctan(\pi/2n)\re(w)\big\}
    \; \; \forall z\in \hpln_+\,.
  \end{equation}
  Furthermore, note that
  \begin{itemize}
    \item[$(*)$] $(\bcdot)^{1/n}$ extends to $\bdy{\hpln_+}$ as a
    continuous function 
    such that $\lim_{\overline{\hpln}_+\ni z\,\to\,\infty}\,z^{1/n} = \infty$.
  \end{itemize}
  
  If $n = 1$ then set $F := f$. If $n\geq 2$, then define
  \[
    F(\langle x_1,\dots, x_n\rangle)\,:=\,{\sf h}^{-1}\!\left(
    \prod\nolimits_{1\leq j\leq n}\Big(\frac{1 + f(x_j)}{1 - f(x_j)}\Big)^{1/n}\right) \; \;
    \forall \langle x_1,\dots, x_n\rangle\in \overline{\Sym^n(X)}.
  \]
  By \eqref{E:branch} we see that $F\in \smoo\big(\overline{\Sym^n(X)}\big)\cap
  \hol(\Sym^n(X))$. By $(*)$ and the properties of $f$ it follows that $F$ is
  a weak peak function for $\Sym^n(X)$ at
  $\langle y_1,\dots, y_n\rangle\in \bdy\Sym^n(X)$.
  
  We have just shown that, whether $n = 1$ or $n\geq 2$, each point in
  $\bdy\Sym^n(X)$ admits a weak peak function for $\Sym^n(X)$. Recall that
  $X^*$ is Stein. It follows from
  Result~\ref{R:s-Car_complt}, by taking $Z = \Sym^n(X^*)$, that
  $\Sym^n(X)$ is strongly Carath{\'e}odory complete.
  In particular, $\Sym^n(X)$ is Kobayashi complete.
  
  By a result of Kiernan \cite{kiernan:rbtthm70}, it follows that
  $\Sym^n(X)$ is taut. 
\end{proof}
\smallskip

\section{Technical propositions}\label{s:technical}

In proving Theorem~\ref{t:bord}, we will need to understand the behaviour of
holomorphic maps $f : Z\to\overline{\Sym^n(Y)}$, $n\geq 2$\,---\,where $Z$ is
connected and $Y$ is as in Theorem~\ref{t:bord}\,---\,in the event that 
${\sf range}(f)\not\subset \Sym^n(Y)$. 

To this end, we shall use the notation introduced in Sections~\ref{s:intro}
and~\ref{s:symm_prelim}.
Thus, given a Riemann surface $Y$ and $(y_1,\dots, y_n)\in Y^n$, $\symq$
is as introduced in Section~\ref{s:intro}, and
\[
  \langle  y_1,\dots, y_n\rangle\,:=\,\symq(y_1,\dots, y_n).
\]
For a point $y\in Y$, a presentation of $y$ having the form of the left-hand side
of the above equation will be called the {\em quotient representation of $y$}.

We require one further observation. For a Riemann surface $Y$, let
$D_1,\dots, D_n$ be non-empty subsets of
$Y$ such that $D_1\times\dots\times D_n$
is not necessarily closed under the $S_n-$action on $Y^n$, $n\geq2$. In
any circumstance, we shall use $\symq(D_1\times\dots\times D_n)$ 
to denote the image of the set $D_1\times\dots\times D_n$ under the
map $\symq : Y^n\to \Sym^n(Y)$.

We begin with the following simple lemma:

\begin{lemma}\label{l:hol}
  Let $X$ be a Riemann surface, $n\geq 2$, and let $D_1,\dots,D_n \subset X$ 
  be open subsets. Write $\dee := \bigcup_{j=1}^n D_j$.
  Define $H :=
  \symq(D_1\times\dots\times D_n)$. Suppose $\phi : \dee\to \C$ is a 
  bounded holomorphic map and $\mathscr{S}$ a symmetric polynomial
  in $n$ indeterminates. Then, the relation
  $\Gamma \subset H\times\C$ defined by
  \begin{multline*}
    \Gamma\,:=\,\big\{\big(\langle  v_1,\dots, v_n\rangle, w\big)\in H\times\C : 
    w = \mathscr{S}(\phi(x_1),\dots , \phi(x_n)) \; \text{and} \\
    (x_1,\dots, x_n)\in (\symq^{-1}\{\langle  v_1,\dots, v_n\rangle\}
    \cap D_1\times\dots\times D_n)\big\}.
  \end{multline*}
  is the graph of a holomorphic function defined on $H$.
\end{lemma}
\begin{proof}
  Let $\pi_1$ (resp., $\pi_2$) denote the projection onto the first (resp., second)
  factor of $H\times \C$.
  Consider any point $\langle  v_1,\dots, v_n\rangle\in H$. That
  $\pi_1^{-1}\{\langle  v_1,\dots, v_n\rangle\}\cap \Gamma$ is a singleton follows
  clearly from the fact that $\mathscr{S}$ is a symmetric polynomial. It is thus
  the graph of a function $\Phi$. Consider the mapping
  $\Phi^\prime: D_1\times\dots\times D_n \to \C$ given by
  \[
    (x_1,\dots,x_n) \mapsto \mathscr{S}(\phi(x_1),\dots,\phi(x_n)).
  \]
  Let us write $\Delta := D_1\times\dots\times D_n$. By construction, 
  $\Phi^\prime = \Phi \circ \big(\!\left.\symq\right|_{\Delta}\big)$. Since
  $\symq$ admits holomorphic branches of local inverses around any of its
  regular values, the above construction shows that
  $\Phi$ is holomorphic outside the set of critical values of
  $\symq$ in $H$. Since $\phi$, and thus $\Phi^\prime$, are bounded,
  Riemann's removable singularities theorem implies that $\Phi$
  is holomorphic on $H$.
\end{proof}

The key result needed is the following, which
generalizes Lemma~5 of \cite{edigarian2005geometry}.

\begin{proposition}\label{p:INDEP}
  Let $Y$ be a connected bordered Riemann surface with $\smoo^2$-smooth
  boundary and let $S$ be the compact Riemann
  surface from which $Y$ is obtained by excising a finite number of
  closed disks. Let $Z$ be a connected complex manifold and
  let $f: Z \to \Sym^n(Y)$, $n\geq 2$, be a holomorphic
  map such that $f(Z) \subset \overline{\Sym^n(Y)}$ (where the closure is in
  $\Sym^n(S)$). Suppose there exists a $z_0\in Z$ such that  
  $f(z_0)$ is of the form $\langle y_1, \boldsymbol{\ast} \rangle$, $y_1\in \bdy Y$.
  Then
  \[
    f(z)\;\;\text{is of the form $\langle y_1, \boldsymbol{\ast} \rangle$  for
    all}\;\;z \in Z.
  \]  
  Moreover, if $y_1$ appears
  $k$ times, $1\leq k\leq n$, in the quotient representation of $f(z_0)$ then the same is
  true for $f(z)$ for all $z\in Z$.
\end{proposition}
\begin{proof}
  Let $y_2,\dots, y_l$ be the other \emph{distinct} points that
  appear in the quotient representation of $f(z_0)$. Let $D_1,\dots,D_l$ 
  be small coordinate disks in $S$ centered at 
  $y_1,\dots,y_l$, respectively, whose closures are pairwise disjoint.
  Let $(D_j, \psi_j)$, $j = 1,\dots, l$ , denote the coordinate charts. By ``coordinate
  disks centered at $y_j$'', we mean that $\psi_j(D_j) = \disk$ and
  $\psi_j(y_j) = 0$, $j = 1,\dots, l$. Furthermore,
  as $Y$ has $\smoo^2$-smooth boundary, we can (by shrinking $D_1$ and scaling
  $\psi_1$ if necessary) ensure that
  \begin{itemize}
    \item $\psi_1(\bdy Y\cap D_1)\cap \{\zt\in \C : |\re(\zt) - 1|^2
    + |\im(\zt)|^2 = 1\}\,=\,\{\psi_1(y_1)\}\,=\,\{0\}$; and
    \item $\psi_1(Y\cap D_1)\subset \{\zt\in \disk : |\re(\zt) - 1|^2
    + |\im(\zt)|^2  > 1\}$.
  \end{itemize}
  Let us define $\phi\in \hol(D_1)$ by
  \[
    \phi(y)\,:=\,\exp\left\{\frac{\psi_1(y)}{2 - \psi_1(y)}\right\} \; \; \; \forall y\in D_1.
  \]
  Using the fact that the M{\"o}bius transformation $\zt\mapsto \zt/(2 - \zt)$ maps
  the circle $\{\zt\in \C : |\re(\zt) - 1|^2 + |\im(\zt)|^2  = 1\}$ onto $\{\zt\in \C : \re(\zt) = 0\}$,
  it is routine to verify that
  \begin{equation}\label{E:peak}
    \phi(y_1) =1 \quad\text{and} \quad |\phi(y)| < 1 \; \; \; 
    \forall y \in D_1 \cap (\overline{Y}\!\setminus\!\{y_1\}),
  \end{equation}
  and that $\phi$ is a bounded function.
  
  Write $\dee := \sqcup_{j=1}^l D_j$ and define a function $\widetilde{\phi}\in \hol(\dee)$
  as follows
  \[
    \widetilde{\phi}(y)\,:=\,\begin{cases}
    					\phi(y), & \text{if $y\in D_1$}, \\
					0, & \text{otherwise}.
				     \end{cases}
  \]
  Let $H := \symq(D_1^k \times D_2^{k_2} \times \dots \times D_l^{k_l})$,
  where $k_j$ is the number of times $y_j$ appears in the quotient
  representation of $f(z_0)$, $j = 2,\dots, l$. Now consider the following relation
  $\Gamma \subset H\times\C$ defined by
  \begin{multline}
    \Gamma\,:=\,\big\{\big(\langle  v_1,\dots, v_n\rangle, w\big)\in H\times\C : 
    w = \widetilde{\phi}(x_1) + \dots + \widetilde{\phi}(x_n) \\
    \text{and} \; (x_1,\dots, x_n)\in (\symq^{-1}\{\langle  v_1,\dots, v_n\rangle\}
    \cap D_1^k \times D_2^{k_2} \times \dots \times D_l^{k_l})\big\}.
  \end{multline}   
  It follows from Lemma~\ref{l:hol} that $\Gamma$ is the graph
  of a  function, say $\Phi$, that is holomorphic on $H$.
  Now as $f(z_0) \in H$ and $H$ is an open neighborhood of $f(z_0)$,
  we can find a small connected open set $U\varsubsetneq Z$ around $z_0$
  such that $f(U) \subset H$. Consider the holomorphic map
  $\Phi\circ \big(\left.f\right|_U\big)$. As $f(U)\subset \overline{\Sym^n(Y)}$, 
  we have, by construction:
  \[
   \Phi\circ \big(\left.f\right|_U\big)(z_0)\,=\,k\,=\,\sup\nolimits_{z\in U}
   \big|\Phi\circ \big(\left.f\right|_U\big)(z)\big|.
  \]
  By the maximum modulus theorem,  $\Phi\circ \big(\left.f\right|_U\big)\equiv k$.
  By the definition of the function $\Phi$, we deduce that the conclusion of our
  proposition holds true on the open set $U$.
  
  Let $E$ be the set of points of $Z$ for which the conclusion of the proposition
  holds true. By hypothesis, $E$ is non-empty. The above argument shows
  that $E$ is an open set. Let $z \in Z\setminus E$. If $y_1$ does not 
  appear in the quotient representation $f(z)$ at all, then, by continuity, there
  exists a neighbourhood $U_z$ of $z$ 
  such that the same is true for every point in $U_z$. On the other hand, if
  $y_1$ does appear in the quotient representation of $f(z)$ but not
  $k$ times, then the argument given prior to this paragraph 
  shows that we can find a neighbourhood $U_z$ of $z$ such that the same is true
  for every point in $U_z$. In either case, therefore, $U_z \subset (Z \setminus E)$.
  This shows that $E$ is closed. Therefore $E = Z$.
\end{proof}
\smallskip

\section{The proof of Theorem~\ref{t:bord}}\label{s:bord}

In proving Theorem~\ref{t:bord} we will find it convenient to use a certain expression,
which we now define.

\begin{definition}\label{d:deps_only}
  Let $M_1,\dots, M_n$ and $N$ be complex manifolds, and $\mathfrak{V}$ a proper (possibly empty)
  analytic subvariety of  $M_1\times\dots\times M_n$. Let 
  $f : (M_1\times\dots\times M_n)\setminus \mathfrak{V}\to N$ be a holomorphic map. We say
  that $f$ {\em depends only on the $j$-th coordinate on
  $(M_1\times\dots\times M_n)\setminus \mathfrak{V}$}, $1\leq j\leq n$,
  if for each $x\in M_j$ lying outside
  some proper (possibly empty) analytic subvariety of $M_j$,
  \begin{multline}
    f(x_1,\dots,x_{j-1}, x, x_{j},\dots, x_{n-1}) = f(y_1,\dots,y_{j-1}, x, y_{j},\dots, y_{n-1}) \\
    \text{for all $(x_1,\dots,x_{n-1})\neq (y_1,\dots,y_{n-1})\in \prod_{i\neq  j}M_i$}
   \end{multline}
   such that $(x_1,\dots,x_{j-1}, x, x_{j},\dots, x_{n-1}), 
   (y_1,\dots,y_{j-1}, x, y_{j},\dots, y_{n-1})\notin \mathfrak{V}$.
\end{definition}

Before we give the proof of Theorem~\ref{t:bord}, we ought to point out to the reader a convention
that will be used below. Given a product space, $\pi_j$ will denote the projection onto the $j$-th coordinate.
If several product spaces occur in a discussion, we {\bf shall not add} additional labels
to $\pi_j$ to indicate the domain of this projection unless there is scope for ambiguity.

\begin{proof}[The proof of Theorem~\ref{t:bord}]
  Let $S$ be a compact Riemann surface such that $Y$ is obtained from
  $S$ by excising a finite number of closed disks such that $\bdy Y$ is
  $\smoo^2$-smooth. Theorem~\ref{t:bord} is a tautology when $n = 1$,
  so it will be understood here that $n\geq 2$.  
  Let $R_j$, $j = 1,\dots,n$, be as in the
  statement of Theorem~\ref{t:bord}.

  Let $p=(p_1,\dots,p_n)$ be a point in $R_1 \times
  \cdots \times R_n$ such that, for each $1 \leq j \leq n$, $p_j$ is a limit point of $R_j\setminus X_j$.
  Also by hypothesis, we can choose $p_j$ to belong to $\bdy{X}_j$.
  Let $(U_j,\psi_j)$ be holomorphic coordinate charts of $R_j$ chosen in such a way that:
  \begin{itemize}
    \item $p_j \in U_j$; and
    \item Each $U_j$ is biholomorphic to a disk.
  \end{itemize}
  Let $W_j := U_j \cap X_j$ and $V_j := \psi_j(W_j)$. 
  For $(z_1,\dots,z_n) \in V_1 \times \dots \times V_n$, let
  \begin{equation}\label{e:g}
    g(z_1,\dots,z_n) :=  F(\psi_1^{-1}(z_1),\dots,\psi_n^{-1}(z_n)).
  \end{equation}
  
  Fix a point $q\in {\bdy X_n}\cap U_n$. Consider a sequence $\{w_\nu\} \subset V_n$ such that
  $w_\nu \to \psi_n(q)$.
  Let
  \[
    \phi_\nu :V_1 \times \dots \times V_{n-1} \to \Sym^n(Y) := 
    g(z_1,\dots,z_{n-1},w_\nu).
  \]
  We claim that we can extract a subsequence $\{w_{\nu_m}\}\subset \{w_\nu\}$ such that
  $\{\phi_{\nu_m}\}$ converges
  uniformly on compacts to a holomorphic mapping $h  :V_1 \times \dots \times
  V_{n-1} \to \Sym^n(S)$. To this end, fix another
  connected bordered Riemann surface, $Y^*\subset S$, with $\smoo^2$-smooth boundary
  such that $Y\Subset Y^*$. 
  By Theorem~\ref{t:symm_taut}, both $\Sym^n(Y)$ and
  $\Sym^n(Y^*)$ are taut. Owing to the tautness of $\Sym^n(Y)$, and
  as $F$ is proper, we can extract a subsequence $\{w_{\nu_m}\}\subset \{w_\nu\}$ such that
  $\phi_{\nu_m}$ is compactly divergent\,---\,i.e., given compacts
  $K_1\subset V_1 \times \dots \times V_{n-1}$ and $K_2\subset \Sym^n(Y)$, there
  exists an integer $M(K_1, K_2)$ such that
  \begin{equation}\label{e:comp_div}
    \phi_{\nu_m}(K_1)\cap K_2\,=\,\emptyset \; \; \;  \forall m\geq M(K_1, K_2).
  \end{equation}
  We now view each $\phi_{\nu_m}$ as a map into $\Sym^n(Y^*)$. This time,
  owing the tautness of $\Sym^n(Y^*)$, there exists a holomorphic map
  $h  :V_1 \times \dots \times V_{n-1} \to \Sym^n(S)$ such that\,---passing to
  a subsequence of $\{\phi_{\nu_m}\}$ and {\bf relabelling} if
  necessary\,---\,$\{\phi_{\nu_m}\}$ converges uniformly on compacts to $h$.
  This establishes our claim. From this and \eqref{e:comp_div} it follows
  that $h(V_1 \times \dots \times V_{n-1}) \subset \bdy\Sym^n(Y)$. It follows from
  Proposition~\ref{p:INDEP} that there exists a point $\xi\in \bdy Y$
  such that
  \begin{equation}\label{e:const}
    h(z)\;\;\text{is of the form $\langle \xi, \boldsymbol{\ast} \rangle$  for
    all}\;\;z \in V_1 \times \dots \times V_{n-1}.
  \end{equation}

  It is a classical fact\,---\,see \cite[Chapter~5]{jost2005riemann}, for instance\,---\,that
  there exists a bounded, non-constant function $\chi$ that is holomorphic on some open
  connected set $\Nb$ that contains $\overline{Y}$. Let $\Psi : \Sym^n(\Nb) \to \C^n$ be
  defined by
  \[
    \langle z_1,\dots,z_n \rangle \longmapsto \big(\elsym{1}(\chi(z_1),\dots,\chi(z_n)),
    \elsym{2}(\chi(z_1),\dots,\chi(z_n)),\dots, \elsym{n}(\chi(z_1),\dots,\chi(z_n))\big).
  \]
  A remark on the purpose of the map $\Psi$ is in order. If we could, by shrinking each
  $U_j$ if necessary, find a single model coordinate chart $(\OM, \varPsi)$ on $\Sym^n(Y)$
  (refer to Section~\ref{s:symm_prelim} for some remarks on the term ``model coordinate chart'')
  so that
   \begin{itemize}
     \item[$i)$] $F(W_1\times\dots\times W_n)\subset \OM$, and
     \item[$ii)$] $\overline{W}_j\cup \bdy X_j$ is indiscrete for each $j = 1,\dots, n$,
   \end{itemize}
   then the principal part of our proof would reduce to an application of 
   \cite[Theorem~1.2]{debraj2015function} by Chakrabarti--Gorai. However, it is far
   from clear that one can find coordinate charts that satisfy {\bf both} $(i)$ and
   $(ii)$. The role of the map $\Psi$ is to compensate for this difficulty.
   \smallskip
  
  \noindent{{\bf Step 1.} {\em Finding local candidates for $F_1,\dots,F_n$}}
  \vspace{0.5mm}
  
  \noindent{The argument at this stage of our proof closely follows that of
  Edigarian--Zwonek \cite{edigarian2005geometry} and Chakrabarti--Gorai 
  \cite{debraj2015function}. But since we must modify the map $g$ (see
  \eqref{e:g} above) in order to use the latter
  argument\,---\,which has consequences on what follows\,---\,we shall present
  parts of this argument in some detail. We begin by defining $G := \Psi\circ g$
  (the need for this map is hinted at by our preceding remarks). 
  By the definition of the map $\Psi$, its holomorphic derivative 
  is non-singular on an open dense subset of
  $\Sym^n(\Nb)$. Therefore\,---\,since $F$ is a proper holomorphic map\,---\,the complex
  Jacobian of $G$ does not vanish identically on
  $V_1 \times \dots \times V_n$. We expand this latter
  determinant along the last column to conclude that there exists $\mu \in
  \{1,\dots,n\}$ such that
  \begin{equation}\label{e:mu}
    \det \left[ \partl{G_i}{z_j}\right]_{i=1,\dots,n,\,i \neq \mu,\,j = 1,\dots,n-1}
    \not \equiv 0 \;\; \text{on $V_1 \times \dots \times V_n$}.
  \end{equation}}
    
  Let us write $\theta := \Psi \circ h$, $\theta^{(m)} := \Psi \circ \phi_{\nu_m}$, $m = 1, 2, 3,\dots$,
  and $\patch := V_1 \times \dots \times V_{n-1}$.
  Owing to \eqref{e:const}, there exists a $C \in \C$ such that
  \begin{equation}\label{e:polybdy}
    C^n - C^{n-1}\theta_1+ \dots + (-1)^{n-1}C\theta_{n-1} + (-1)^n \theta_n\,\equiv\,0 \; \;\text{on 
    $\patch$},
  \end{equation}
  where $\theta = (\theta_1,\dots,\theta_n)$. Differentiating with respect to $z_j$,
  $j = 1,\dots,n-1$, we get 
  \begin{equation}\label{e:sys1}
    -C^{n-1} \partl{\theta_1}{z_j} + \dots + (-1)^{n-1} C \partl{\theta_{n-1}}{z_j}
    + (-1)^n \partl{\theta_n}{z_j}\,\equiv\,0 \;\; \text{on $\patch$}.
  \end{equation}
  Rearranging \eqref{e:sys1}, we get the following system of $(n-1)$ equations:
  \begin{equation}\label{e:sys2}
    \sum_{k =1,\dots,n,\,k \neq \mu} (-1)^kC^{n-k}\partl{\theta_k}{z_j}\,=\,(-1)^{\mu+1}C^{n-\mu}
    \partl{\theta_\mu}{z_j} \;\; \text{on $\patch$}, \;\; j = 1,\dots,n-1. 
  \end{equation}
   
  Given an $(n-1)\times n$ matrix $B$ and 
  $l \in\{1,\dots,n\}\setminus\{\mu\}$,
  denote by $\Delta_l(B)$ the determinant of the
  $(n-1) \times (n-1)$ matrix obtained by:
  \begin{itemize}
    \item deleting the $\mu$-th column of $B$; and
    \item replacing the $l$-th column by the $\mu$-th column of $B$.
  \end{itemize}
  Denote by $\Delta_\mu(B)$ the
  determinant of the  $(n-1)\times (n-1)$ matrix obtained by deleting the
  $\mu$-th column of $B$. Note that each of the functions $\Delta_j$ is a
  polynomial in the entries of the matrix $B$. 
  
  We now introduce the $(n-1)\times n$ matrices
  \[
    D_{n-1}\theta(z^\prime)\,:=\,\left[
    \partl{\theta_k}{z_j}(z^\prime)\right]_{1\leq j\leq n-1,\,1\leq k\leq n} \; \; \text{and}
    \; \;
    {}^m\!D_{n-1}\theta(z^\prime)\,:=\,\left[
    \partl{\theta^{(m)}_k}{z_j}(z^\prime)\right]_{1\leq j\leq n-1,\,1\leq k\leq n},
  \]
  where $z^\prime := (z_1,\dots,z_{n-1})$. We also set:
  \[
    \mathfrak{A}\,:=\,\big\{z^\prime\in \patch : \Delta_\mu\big(D_{n-1}\theta(z^\prime)\big) = 0\big\}.
  \]
  
  Depending on $\mathfrak{A}$, we need to consider two cases.
  
  \noindent{{\bf Case 1.} $\mathfrak{A}\varsubsetneq \patch$.}
  \vspace{0.5mm}

  \noindent{By applying Cramer's rule to the system described by \eqref{e:sys2},
  we get:
  \[
    (-1)^l C^{n-l}\,=\,(-1)^\mu C^{n-\mu}\frac{\Delta_l(D_{n-1}\theta)}{\Delta_\mu
    (D_{n-1}\theta)} \;\; \text{on $(\patch\!\setminus\!\mathfrak{A})$ and} \;\; l\in 
    \{1,\dots,n\}\!\setminus\!\{\mu\}.
  \]
  If $\mu \neq 1$, we shall argue by taking $l = \mu-1$ in the above. If $\mu = 1$,
  we shall take $l = 2$. We
  shall first consider the case $\mu\neq 1$. In this case, the above equation gives
  \begin{equation}\label{e:relation}
    -C \Delta_\mu(D_{n-1}\theta) = \Delta_{\mu-1}(D_{n-1}\theta) \;\; \text{on $\patch$}. 
  \end{equation}}
  
  \noindent{{\bf Case 2.} $\mathfrak{A} = \patch$.}
  \vspace{0.5mm}
   
  \noindent{As in Case~1, we assume $\mu\neq 1$. Since the system \eqref{e:sys2}\,---\,treating $C$
  as the indeterminate\,---\,admits a solution, $\Delta_{\mu}(D_{n-1})\equiv 0$ forces on us the
  conclusion \eqref{e:relation} for trivial reasons.}
  
  So, in each of the above cases, we get the identity \eqref{e:relation}.
  Differentiating this identity with respect to $z_j$ and eliminating $C$, we get
  the relations
  \[
    \Delta_\mu(D_{n-1}\theta) \partl{\Delta_{\mu-1}(D_{n-1}\theta)}{z_j} - \Delta_{\mu-1}(D_{n-1}\theta)
    \partl{\Delta_\mu(D_{n-1}\theta)}{z_j}\,\equiv\,0 \;\; \text{on $\patch$}, \;\; j = 1,\dots,n-1. 
  \]
  The left-hand sides of the above relations are constituted of polynomial expressions involving
  \[
    \lim_{m\to \infty}g_s\left(z_1,\dots,z_{n-1},  w_{\nu_m}\right), \; \; s = 1,\dots,n,
  \]
  their compositions with the function $\chi$, 
  and their partial derivatives (with respect to $z_1,\dots,z_{n-1}$) up to order two. Hence, by
  Weierstrass's theorem on the derivatives of holomorphic functions, we have
  \begin{align}
    \lim_{m\to \infty}&\Delta_\mu({}^m\!D_{n-1}\theta)(z^\prime)
    \partl{\Delta_{\mu-1}({}^m\!D_{n-1}\theta)}{z_j}(z^\prime)
    - \Delta_{\mu-1}({}^m\!D_{n-1}\theta)(z^\prime)
    \partl{\Delta_\mu({}^m\!D_{n-1}\theta)}{z_j}(z^\prime) \notag \\
    =&\,\Delta_\mu(D_{n-1}\theta)(z^\prime) \partl{\Delta_{\mu-1}(D_{n-1}\theta)}{z_j}(z^\prime)
    - \Delta_{\mu-1}(D_{n-1}\theta)(z^\prime)
    \partl{\Delta_\mu(D_{n-1}\theta)}{z_j}(z^\prime) \notag \\
    =&\,0 \; \; \; \forall z^\prime\in \patch, \; \; j = 1,\dots,n-1. \label{e:lim_0}
  \end{align}
  
  Consider the functions $\tau_j : V_1 \times \dots \times V_{n} \to \C$ defined as follows:
  \begin{align*}
    \tau_j(z^\prime, z_n)\,:=\,\Delta_\mu & (D_{n-1}G)(z^\prime, z_n)
    \partl{\Delta_{\mu-1}(D_{n-1} G)}{z_j}(z^\prime, z_n)\\
    & - \Delta_{\mu-1}(D_{n-1} G)(z^\prime, z_n)
         \partl{\Delta_\mu(D_{n-1}G)}{z_j}(z^\prime, z_n),
  \end{align*}
  for each $j = 1,\dots,n-1$. Here, $D_{n-1}G(\bcdot, z_n)$ is an $(n-1)\times n$ matrix
  that is defined in the same way as $D_{n-1}\theta$.
  Observe that \eqref{e:lim_0} holds for any subsequence 
  $\{w_{\nu_m}\}\subset \{w_\nu\}$ with the properties discussed right after \eqref{e:g},
  where $V_n\ni w_\nu\to q$. Finally, as $q\in \bdy X\cap U_n$ was picked arbitrarily,
  \eqref{e:lim_0} implies that
  \[
    \tau_j(z^\prime, \zt)\longrightarrow 0 \; \; \text{as $\zt\to \psi(U_n)\cap \bdy V_n$
    for each $z^\prime\in \patch$},
  \]
  and for each $j = 1,\dots,n-1$. 
  Thus we can extend each $\tau_j$ to a continuous function $\widetilde{\tau}_j$
  defined on $V_1 \times \dots \times V_ {n-1} \times \psi(U_n)$ by setting
  $\widetilde{\tau}_j(z^\prime, z_n) = 0$ whenever $z_n \in \psi(U_n) \setminus V_n$. By Rado's
  theorem\,---\,see \cite[Chapter~4]{narasimhan1971SCV}\,---\,$\widetilde{\tau}_j$ is holomorphic
  on $V_1 \times \dots \times V_{n-1} \times \psi(U_n)$. Let us now fix
  $z^\prime\in \patch$ and $j\,:\,1\leq j\leq n-1$. By construction,
  $\psi(U_n) \setminus V_n$ has at least one limit point in $\psi(U_n)$. Thus, by the identity
  theorem, $\widetilde{\tau}_j(z^\prime, \bcdot)$ is identically $0$. As this holds true
  for every $z^\prime$ and $j$, it follows that each $\tau_j$ is identically $0$.
  
  Set 
  \[
    \gamma_n := -\frac{\Delta_{\mu-1}(D_{n-1} G)}{\Delta_\mu(D_{n-1}
    G)}.
  \]
  The function $\gamma_n$ is well-defined on the set 
  $(V_1 \times \dots \times V_n)\setminus \mathcal{A}$, where 
  \[
    \mathcal{A}\,:=\,\left\{z \in V_1 \times \dots \times V_n: \Delta_\mu\big(D_
    {n-1}G(z)\big) = 0\right\}.
  \]
  By \eqref{e:mu}, $\mathcal{A}$ is a proper analytic subvariety of
  $V_1 \times \dots \times V_n$. Observe that
  $\left.\tau_j\right|_{(V_1 \times \dots \times V_n)\setminus \mathcal{A}}$ is
  the numerator of  $\partl{\gamma_n}{z_j}$, whence
  \[
    \partl{\gamma_n}{z_j}\,\equiv\,0 \;\; \text{on $(V_1 \times \dots \times V_n)\setminus
    \mathcal{A}$},
  \]
  for $j=1, \dots, n-1$. Since $\mathcal{A}$ is a proper analytic subvariety, this implies
  that $\gamma_n$ depends only on 
  $z_n$\,---\,in the sense of Definition~\ref{d:deps_only}\,---\,on each
  set of the form $\mathscr{M}\setminus \mathcal{A}$, where
  $\mathscr{M}$ is a connected component of $V_1 \times \dots \times V_n$.
  
  Appealing to \eqref{e:relation}, and arguing in the same manner as above, we get
  \[
    \gamma_n(z^\prime, \zt)\longrightarrow C \; \; \text{as $\zt\to \psi(U_n)\cap \bdy V_n$
    for each $z^\prime\in \patch$},
  \]
  where, we now recall, $C$ satisfies the equation \eqref{e:polybdy}.
  Again, by an argument involving Rado's
  theorem\,---\,see \cite{edigarian2005geometry} or \cite{debraj2015function}\,---\,that
  is analogous to the one above, it follows that
  \begin{multline} \label{eq:gamma}
    \gamma_n^n(z) - \gamma_n^{n-1}(z)G_1(z) + \dots +
    (-1)^{n-1}\gamma_n(z)G_{n-1}(z) \\
    + (-1)^n G_n(z)\,\equiv\,0 \; \; \forall z\in 
    (V_1 \times \dots \times V_n)\setminus \mathcal{A},
  \end{multline}
  where we write $G = (G_1,\dots, G_n)$. This
  shows that $\gamma_n((V_1 \times \dots \times V_n) \setminus \mathcal{A})
  \subset \chi(Y)$ which, by the choice of $\chi$, is bounded. 
  By Riemann's removable singularities theorem,
  $\gamma_n$ extends to be holomorphic on $V_1\times\dots \times V_n$. 

  A completely analogous argument can be given\,---\,which results in a slightly
  different expression for $\gamma_n$\,---\,when $\mu = 1$ (in which case, we take
  $l = 2$, $l$ as introduced at the beginning of Step~1).

  Repeating this argument with some $i$ replacing $n$ above 
  yields us maps $\gamma_i: V_1\times\dots \times V_n\to \C$ that satisfy equations
  analogous to \eqref{eq:gamma}.
  What we have at this stage is summarized by 
  the following commutative diagram:
  
  \begin{figure}[ht]
    \begin{tikzcd}
      X_1\times\dots\times X_n \arrow[r, "F"]
      & \Sym^n(Y) \arrow[rr, "\Psi"]
      &
      & \Cn \\
      & Y^n \arrow[rr, "{(\chi\circ \pi_1,\dots,\,\chi\circ\pi_n)}"]
         \arrow[u, "\symq"]
      &  
      & \Cn \arrow[u, "{\pi^{(n)}}"] \\
      & & \\
      V_1\times\dots\times V_n \arrow[uuu, "{(\psi^{-1}_1\circ\pi_1,\dots,\,\psi^{-1}_n\circ\pi_n)}"]
      \arrow[uurrr, "{(\gamma_1,\dots,\gamma_n)}"]
    \end{tikzcd}
  \end{figure}
  \noindent{where we use $\pi_j$, $j = 1,\dots,n$, to denote the projection onto the $j$-th coordinate
  (where the product domain in question is understood from the context). Let us write:
  \begin{align*} 
    \mathscr{C}\,&:=\,\text{the set of critical points of $\symq : Y^n \to \Sym^n(Y)$}, \\
    \mathscr{C^*}\,&:=\,\text{the set of critical points of $\pi^{(n)} : \Cn \to \Cn$}. 
  \end{align*}
  Here $\pi^{(n)}$ is as introduced in Section~\ref{s:intro}.
  We now find connected open sets
  $W^*_j\subset W_j$, $j = 1,\dots, n$, that are so small that:
  \begin{itemize}
    \item[$a)$] $F(W^*_1\times\dots\times W^*_n)\cap \symq(\mathscr{C})\cap
    \Psi^{-1}\big(\pi^{(n)}(\mathscr{C}^*)\big) = \emptyset$;
    \vspace{0.5mm}
    
    \item[$b)$] $\pi^{(n)}$ is invertible on $\Psi\big(F(W^*_1\times\dots\times W^*_n)\big)$; and
    \vspace{0.5mm}
    
    \item[$c)$]  The map $(\chi\circ\pi_1,\dots, \chi\circ\pi_n)$ is invertible on each image
    of $\Psi\big(F(W^*_1\times\dots\times W^*_n)\big)$ under a branch of a local inverse of
    $\pi^{(n)}$ that intersects the image of $(\chi\circ\pi_1,\dots, \chi\circ\pi_n)$.
  \end{itemize} 
  Let $\big(\pi^{(n)}\big)^{-1}_s$, $s = 1,\dots, n!$, denote the branches introduced in $(c)$. 
  The definition of the map $\Psi$ ensures that, in fact, the images of
  $\Psi\big(F(W^*_1\times\dots\times W^*_n)\big)$ under each $\big(\pi^{(n)}\big)^{-1}_s$
  are contained in $(\chi\circ\pi_1,\dots, \chi\circ\pi_n)(Y)$. From this
  and a routine diagram-chase\,---\,since, by construction, the arrow representing
  $\pi^{(n)}$ can be reversed on $\Psi\big(F(W^*_1\times\dots\times W^*_n)\big)$\,---\,we
  see that there exists a number $s^0$, $1\leq s^0\leq n!$ such that
  \[
    (\gamma_1,\dots, \gamma_n)\circ \big(\psi_1(W^*_1)
    \times\dots\times \psi_n(W^*_n)\big)\,=\,\big(\pi^{(n)}\big)^{-1}_{s^0}\big(
    \Psi(F(W^*_1\times\dots\times W^*_n))\big).
  \]
  Thus, by $(c)$, there is a local holomorphic inverse\,---\,call it $\mathscr{I}
  \equiv (\mathscr{I}_1,\dots, \mathscr{I}_n)$\,---\,of
  $(\chi\circ\pi_1,\dots, \chi\circ\pi_n)$ such that the maps
  \[
    {\sf f}_j\,:=\,\mathscr{I}_j\circ (\gamma_1,\dots, \gamma_n)\circ (\psi_1\circ\pi_1,
    \dots, \psi_n\circ\pi_n)
  \]
  are well-defined on $W^*_1\times\dots\times W^*_n$  and holomorphic, $j = 1,\dots,n$.
  From the above commutative diagram, we see that
  \[
    \left.F\right|_{W^*_1\times\dots\times W^*_n} = \symq\circ({\sf f}_1,\dots, {\sf f}_n).
  \]}
  
  Since, by construction, $\psi_1(W^*_1)\times\dots\times \psi_n(W^*_n)$ lies in a connected
  component of $V_1\times\dots\times V_n$,  $\gamma_j$ depends only on $z_j$ 
  on $\psi_1(W^*_1)\times\dots\times \psi_n(W^*_n)$ for each
  $j = 1,\dots, n$. Then, owing to the structure of the map $(\chi\circ\pi_1,\dots, \chi\circ\pi_n)$,
  of which $\mathscr{I}$ is a local inverse, it follows that
  \begin{multline}\label{e:indep}
    \text{for each $j$, $j = 1,\dots,n$, the map ${\sf f}_j : W^*_1\times\dots\times W^*_n
    \to Y_j$ depends} \\
    \text{only on the $j$-th coordinate on $W^*_1\times\dots\times W^*_n$.}
  \end{multline}
  
  \noindent{{\bf Step 2.} {\em Establishing the (global) existence of $F_1,\dots,F_n$}}
  \vspace{0.5mm}
  
  \noindent{We abbreviate $X_1\times\dots\times X_n$ to $X$. Let
  $\excep := F^{-1}\big(\symq(\mathscr{C})\big)$ ($\mathscr{C}$ is as introduced
  above), which
  is a proper analytic subset of $X$. If $x \in X\setminus \excep$, then
  we can find a connected product neighbourhood $\OM_x$ of $x$ such that the map
  $F|_{\OM_x}$ lifts to $Y^n$ (i.e., it admits a holomorphic map
  $f : \OM_x\to Y^n$ such that $\symq\circ f = F|_{\OM_x}$).}
  
  Fix an $x_0 \in X\setminus \excep$. Consider any path 
  $\Gamma: [0,1] \to X\setminus \excep$ such that $\Gamma(0)$ is in $W^*_1 \times
  \dots \times W^*_n$ and $\Gamma(1) = x_0$. Here,
  $W^*_j\subset X_j$, $j = 1,\dots, n$, are the domains introduced towards the end of the
  argument in Step~1.
  We can cover $\Gamma([0,1])$ by finitely many product neighbourhoods\,---\,call
  them $\OM^0, \OM^1, \dots \OM^s$\,---\,on which the map $F$ lifts to $Y^n$.
  Moreover, it is easy to see that we can find
  $\OM^0, \OM^1, \dots \OM^s$ and lifts $(f_1^i,\dots, f_n^i) : \OM^i \to Y^n$
  of $F|_{\OM^i}$ to $Y^n$ for each $i$ such that:
  \begin{itemize}
    \item $\OM^0 = W^*_1 \times\dots \times W^*_n$;
    \vspace{0.5mm}
    
    \item $(f_1^0,\dots, f_n^0) : \OM^0 \to Y^n$ is the map $({\sf f}_1,\dots, {\sf f}_n)$
    provided by Step~1;
    \vspace{0.5mm}
    
    \item $\OM^i\cap \OM^{i-1} \neq \emptyset$ for $i = 1,\dots,s$;
    \vspace{0.5mm}
    
    \item For each $i = 1,\dots,s$,
    $f^{i}_j\big|_{K^i} \equiv f^{i-1}_j\big|_{K^i}$ for each $j = 1,\dots,n$,
    where $K^i$ is some connected component of $\OM^{i}\cap \OM^{i-1}$.
  \end{itemize}
  Then, owing to \eqref{e:indep}, it follows from the identity theorem and
  induction that each $f^s_j$ depends only on the $j$-th coordinate on $\OM^s$.
  
  In short, given any $x_0 \in X \setminus\excep$, we can find a
  product neighborhood $N = N_1 \times \dots \times N_n\ni x_0$ and maps
  $f_j : N \to Y$ that depend only on the $j$-th coordinate on $N$
  such that $F|_N = \symq\circ (f_1,\dots,f_n)$.
  \vspace{1mm}
  
  \noindent{{\bf Claim.} {\em This $(f_1,\dots, f_n)$ does not depend on the choice of path 
  $\Gamma$ joining $x_0$ to $W^*_1 \times\dots \times W^*_n$ or the choice of $\OM^i$,
  $i = 1,\dots,s$, covering $\Gamma([0,1])$.}}
  \vspace{0.5mm}
  
  \noindent{To see this, suppose $(\varphi_1,\dots, \varphi_n)$ is a lift of $F$ to $Y^n$ on a
  neighbourhood of $x_0$ obtained by carrying out the above procedure along some different path
  or via a different cover of $\Gamma([0,1])$. Then, there exists a permutation
  $\rho$ of $\{1,\dots,n\}$ such that
  \[
    (\varphi_1,\dots, \varphi_n)\,\equiv\,(f_{\rho(1)},\dots, f_{\rho(n)}) \; \; \text{on a
    neighbourhood of $x_0$}.
  \]
  Now, $\varphi_j$ depends {\bf only} on the $j$-coordinate. The above equation implies
  that $\varphi_j$ depends only on the $\rho(j)$-th coordinate, $j = 1,\dots,n$. This
  is impossible unless $\rho$ is the identity permutation. Hence the claim.}
  \vspace{1mm}
  
  Since the $x_0\in X \setminus\excep$ mentioned above is completely arbitrary, it follows
  from the above Claim that we have holomorphic maps $\widetilde{F}_j : 
  X\setminus \excep \to Y$, $j = 1,\dots n$, such that
  $\widetilde{F}_j$ depends only on the $j$-th coordinate on $X \setminus\excep$ (in
  the sense of Definition~\ref{d:deps_only})
  and such that
  \begin{equation}\label{e:prelim_rel}
    F|_{X\setminus \excep} = \symq\circ(\widetilde{F}_1,\dots, \widetilde{F}_n).
  \end{equation}
   

  By Lemma~\ref{l:embed}, $Y$ is hyperbolically embedded in $S$ ($S$ is as
  introduced at the beginning of this proof). Then, by
  Results~\ref{r:kwack}~and~\ref{r:merhyp}, each $\widetilde{F}_j$ extends
  to a holomorphic map on $X_j$, $j = 1,\dots, n$. By continuity, we can now
  view these extended maps as holomorphic maps $F_j : X_j\to Y$. In view of
  \eqref{e:prelim_rel}, we have our result. The properness of each of the maps
  $F_j$ is straightforward.
\end{proof}

\section*{Acknowledgements}
A part of the research in this paper was conducted during a visit by Gautam Bharali
and Indranil Biswas to the International Centre for Theoretical Sciences (ICTS) to
participate in the meeting {\em Complex Geometry} in March 2017
(programme code: {\bf ICTS/Prog-compgeo/2017/03}). They thank the ICTS for
its support. Gautam Bharali is supported by a Swarnajayanti Fellowship (grant
no.~DST/SJF/MSA-02/2013-14). Indranil Biswas is supported by a J.C.~Bose Fellowship.
Divakaran Divakaran is supported by a project grant from the Indian Institute of
Science Education and Research Bhopal (grant no.~IISERB/INS/MATH/2016091).
Jaikrishnan Janardhanan is supported by a DST-INSPIRE fellowship from the Department
of Science and Technology, India.

\bibliographystyle{amsalpha} \bibliography{bordered}

\end{document}